\documentclass[12pt]{amsart}

\usepackage{pinlabel}
\usepackage{amsmath,amsfonts,amssymb,latexsym,mathrsfs,stmaryrd}
\usepackage{color}
\usepackage{graphicx}
\usepackage{graphics}
\usepackage{enumerate}
\usepackage{amscd} 
\usepackage{amsxtra}
\usepackage{epic,eepic}
\usepackage{hyperref} 
\usepackage{enumitem}
\usepackage{MnSymbol}
\usepackage{dsfont}
\usepackage{tikz-cd}
\usepackage[all]{xy}

\newtheorem{theorem}{Theorem}[section]

\newtheorem{proposition}[theorem]{Proposition}
\newtheorem{corollary}[theorem]{Corollary}

\theoremstyle{definition} 

\newtheorem{definition}[theorem]{Definition}

\newtheorem{remark}[theorem]{Remark}

\newtheorem{notation}[theorem]{Notation} 
\renewcommand{\P}{\bb{P}} 
\newcommand{\C}{\bb{C}}

\newcommand{\Q}{\bb{Q}}

\newcommand{\qu}{/\kern-.7ex/}
\newcommand{\lqu}{\backslash \kern-.7ex \backslash}
\newcommand{\on}{\operatorname} 
\newcommand{\Aut}{\on{Aut}}

\newcommand{\rank}{\on{rank}}

\newcommand{\bb}[1]{\ensuremath{\mathbb{#1}}}

\title[Stacky Double Ramification Cycles]{Double ramification cycles on the moduli spaces of admissible covers}

\author{Hsian-Hua Tseng}
\address{Department of Mathematics\\ Ohio State University\\ 100 Math Tower, 231 West 18th Ave.\\Columbus\\ OH 43210\\ USA}
\email{hhtseng@math.ohio-state.edu}

\author{Fenglong You}
\address{Department of Mathematical and Statistical Sciences, 632 CAB, University of Alberta, Edmonton AB T6G 2G1 Canada.} \email{fenglong@ualberta.ca}


\begin{document}
\date{\today}

\begin{abstract} 
We derive a formula for the virtual class of the moduli space of rubber maps to $[\P^1/G]$ pushed forward to the moduli space of stable maps to $BG$. As an application, we show that the Gromov-Witten theory of $[\P^1/G]$ relative to $0$ and $\infty$ are determined by known calculations.
\end{abstract}

\maketitle 

\tableofcontents
\section{Introduction}
\subsection{$\mathbb{P}^1$-stacks}
This paper is motivated by the study of Gromov-Witten theory of $\mathbb{P}^1$-stacks of the following form:
\[
[\mathbb P^1/G].
\]
Here, $G$ is a finite group. The $G$-action on $\mathbb P^1$ is given by the one-dimensional representation $L=\mathbb C$,
\[
\varphi:G\longrightarrow \mu_a=\on{Im} \varphi\subset \mathbb C^*=GL(L),
\]
together with the trivial one-dimensional representation $\mathbb C$ via 
\[
\mathbb P^1=\mathbb P(L\oplus \mathbb C).
\]
The $\mathbb C^*$-action on $\mathbb P^1$ given by
\[
\lambda\cdot[z_0,z_1]:=[z_0,\lambda z_1], \quad \lambda\in \mathbb C^*, [z_0,z_1]\in \mathbb{P}^1
\]
commutes with this $G$-action and induces a $\mathbb C^*$-action on $[\mathbb P^1/G]$.

\subsection{Stacky rubbers}
The \emph{relative} Gromov-Witten theories of the pairs 
\begin{equation}\label{eqn:rel_pairs}
([\mathbb P^1/G],[0/G]), ([\mathbb P^1/G],[0/G]\cup [\infty/G])
\end{equation} 
arise naturally in the pursue of Leray-Hirsch type results in orbifold Gromov-Witten theory, see \cite{TY}. Indeed, 
\[
[\mathbb P^1/G]=[\mathbb P(L\oplus \mathbb C)/G]\longrightarrow BG
\]
can be viewed as the stacky $\mathbb P^1$-bundle associated to the line bundle 
\[
L\longrightarrow BG.
\]
The relative Gromov-Witten theories of the pairs (\ref{eqn:rel_pairs}) may be studied using virtual localization with respect to the $\mathbb C^*$-action on $[\mathbb P^1/G]$. Rubber invariants naturally arise in this approach. Let
\[
\overline{M}_{g,I}([\mathbb P^1/G],\mu_0,\mu_\infty)^\sim
\]
be the moduli space of rubber maps, see Section \ref{sec:dr_cycle} for precise definitions. Post-composition with $[\mathbb P^1/G]\rightarrow BG$ defines a map
\[
\epsilon:\overline{M}_{g,I}([\mathbb P^1/G],\mu_0,\mu_\infty)^\sim \longrightarrow \overline{M}_{g,l(\mu_0)+l(\mu_\infty)+\#I}(BG).
\]
The cycle 
\[
DR_g^G(\mu_0,\mu_\infty,I):=\epsilon_*[\overline{M}_{g,I}([\mathbb P^1/G],\mu_0,\mu_\infty)^\sim]^{\text{vir}}\in A^g(\overline{M}_{g,l(\mu_0)+l(\mu_\infty)+\#I}(BG))
\]
is termed {\em stacky double-ramification cycle}. The main result of this paper is a formula for $DR_g^G(\mu_0,\mu_\infty,I)$. The formula, which involves complicated notation, is given in Theorem \ref{dr-cycle-theorem} below. 

When $G=\{1\}$, our formula reduces to Pixton's formula for double ramification cycles, proven in \cite{JPPZ}. Our proof, given in the bulk of this paper, closely follows that of \cite{JPPZ}.

The main application of the formula for $DR_g^G(\mu_0,\mu_\infty,I)$ is the following
\begin{theorem}\label{thm:app}
The relative Gromov-Witten theories of 
\[
([\mathbb P^1/G],[0/G]) \text{ and } ([\mathbb P^1/G],[0/G]\cup [\infty/G])
\]
are completely determined.
\end{theorem}

\begin{proof}
Since evaluation maps on $\overline{M}_{g,I}([\mathbb P^1/G],\mu_0,\mu_\infty)^\sim$ factor through $\epsilon$, rubber invariants\footnote{Following \cite{MP}, we treat disconnected invariants as products of connected ones.}
 are all determined by the formula for $DR_g^G(\mu_0,\mu_\infty,I)$, together with the Gromov-Witten theory of $BG$ solved by \cite{jk}.

Virtual localization reduces the calculation of relative Gromov-Witten invariants to calculating rubber invariants with target descendants.
By rubber calculus in the fiber class case \cite{MP}, rubber invariants with target descendants are determined by those without target descendants. The proof is complete.
\end{proof}
Theorem \ref{thm:app} is an evidence supporting \cite[Conjecture 2.2]{TY}, and we expect that Theorem \ref{thm:app} plays an important role in the general case.

\subsection{Acknowledgment}
We thank F. Janda, R. Pandharipande, and D. Zvonkine for discussions. H.-H. T. is supported in part by NSF grant DMS-1506551 and Simons Foundation Collaboration Grant.

\section{Stacky double ramification cycle}\label{sec:dr_cycle}
Let $G$ be a finite group and $L=\mathbb C$  a one dimensional $G$-representation given by the map
\[
G\stackrel{\varphi}{\longrightarrow}\mu_a=\on{Im} \varphi\subset GL(L)=\mathbb C^*.
\]
Let $K:=\ker \varphi$, we obtain the exact sequence
\[
1\longrightarrow K \longrightarrow G \stackrel{\varphi}{\longrightarrow} \mu_a\longrightarrow 1.
\]

\begin{definition}
For a conjugacy class $\mathfrak c \subset G$, define $$r(\mathfrak c)\in \mathbb N$$ to be the order of any element of $\mathfrak c$.
Define\footnote{$a_\mathfrak{c}(L)$ is well-defined because $L$ is $1$-dimensional.} $$a_{\mathfrak c}(L)\in \{0,\ldots, r(\mathfrak c)-1\}$$ to be the unique integer such that each element of $\mathfrak c$ acts on $L$ by multiplication by $\exp \left(\frac{2\pi\sqrt{-1}a_{\mathfrak c}(L)}{r(\mathfrak c)}\right)$. In other words, the representation $\varphi: G\rightarrow GL(L)=\mathbb C^*$ maps $\mathfrak c$ to $\exp \left(\frac{2\pi\sqrt{-1}a_{\mathfrak c}(L)}{r(\mathfrak c)}\right)$.
\end{definition}

Consider the quotient stack $[\mathbb P^1/G]$, where the $G$-action on $\mathbb P^1$ is given by the $1$-dimensional representation $\varphi$ together with the trivial one-dimensional representation $\mathbb C$ via 
\[
\mathbb P^1=\mathbb P(L\oplus \mathbb C).
\]
\begin{definition}\label{defn:rel_data}
Let $A$ denote the following data:
\[
\mu_0=\{(c_{0i},f_{0i},\mathfrak c_{0i})\}_i \quad \mu_{\infty}=\{(c_{\infty i},f_{\infty i},\mathfrak c_{\infty i})\}_i, \quad I=\{\mathfrak c_1,\ldots, \mathfrak c_k\},
\]
where $c_{0i}, c_{\infty i}\in \mathbb{Z}_{\geq 0}$,  $f_{0i}, f_{\infty i}\in \mathbb{N}$, $\mathfrak c_{0i}, \mathfrak c_{\infty i}, \mathfrak c_1,\ldots, \mathfrak c_k  \in \on{Conj}(G)$
such that 
\begin{enumerate}
\item
$f_{0i}$ (resp. $f_{\infty i}$) is the order of any element in $\mathfrak{c}_{0 i}$ (resp. $\mathfrak{c}_{\infty i}$).
\item
$\sum_i \frac{c_{0i}}{f_{0i}}=\sum_j \frac{c_{\infty j}}{f_{\infty j}}$.

\item
$\on{age}_{\mathfrak c_{0i}}(L)=\langle \frac{c_{0i}}{f_{0i}}\rangle, \quad \on{age}_{\mathfrak c_{\infty j}}(L)=\langle \frac{c_{\infty j}}{f_{\infty j}}\rangle$.

\item
$\on{age}_{\mathfrak c_i}(L)=0$, $1\leq i \leq k$. So $\mathfrak c_i \in \on{Conj}(K)$.

\item
Monodromy condition\footnote{For a conjugacy class $(g)$, $(g)^{-1}$ stands for the conjugacy class $(g^{-1})$.} in genus $g$ holds for $\{\mathfrak c_{0i}\} \cup \{\mathfrak c_{\infty j}^{-1}\} \cup \{\mathfrak c_1,\ldots, \mathfrak c_k\}$.
\end{enumerate}
\end{definition}
Here the monodromy condition in genus $g$ means the following.
\begin{definition}[Monodromy condition]\label{defn:monodromy}
Let $H$ be a finite group. We say that the collection of conjugacy classes $\mathfrak{c}_1,...,\mathfrak{c}_n$ of $H$ satisfy monodromy condition in genus $g$ if there exist $$h_i\in \mathfrak{c}_i, 1\leq i\leq n, \quad a_j, b_j\in H, 1\leq j\leq g,$$
such that $$\prod_{i=1}^nh_i=\prod_{j=1}^g [a_j, b_j].$$
\end{definition}

\begin{remark}
The data $\mu_0, \mu_\infty$ are referred to as {\em stacky partitions}. The length of $\mu_0$, denoted by $l(\mu_0)$, is the number of triples in the partition $\mu_0$.
\end{remark}

The moduli space 
\[
\overline{M}_{g,I}([\mathbb P^1/G],\mu_0,\mu_\infty)^\sim
\]
parametrizes stable relative maps of connected twisted curves of genus $g$ to rubber with ramification profiles $\mu_0,\mu_\infty$ over $[0/G]$ and $[\infty/G]$ respectively, and additional marked points whose stack structures are described by $I$. As noted in \cite[Appendix A. 2]{TY}, rubber theory in the stack setting may be defined in the same way as e.g. \cite[Section 1.5]{MP}.

Set $n=l(\mu_0)+l(\mu_\infty)+\#I$. A Riemann-Roch calculation\footnote{Note that the relative tangent bundle $T_{[\P^1/G]}(-[0/G]-[\infty/G])$ entering the Riemann-Roch formula is in fact trivial.} shows that the virtual dimension of $\overline{M}_{g,I}([\mathbb P^1/G],\mu_0,\mu_\infty)^\sim$ is 
$$\text{vdim}\,\overline{M}_{g,I}([\mathbb P^1/G],\mu_0,\mu_\infty)^\sim=2g-3+n.$$  

The moduli space $\overline{M}_{g,n}(BG)$ of $n$-pointed genus $g$ stable maps to $BG$ is smooth of dimension $3g-3+n$. There is a morphism
\[
\epsilon:\overline{M}_{g,I}([\mathbb P^1/G],\mu_0,\mu_\infty)^\sim \longrightarrow \overline{M}_{g,n}(BG)
\]
defined by post-composition with $[\mathbb P^1/G]\rightarrow BG$.
\begin{definition}
The stacky double ramification cycle is defined to be the push-forward
\[
DR_g^G(A)=\epsilon_* [\overline{M}_{g,I}([\mathbb P^1/G],\mu_0,\mu_\infty)^\sim]^{\text{vir}}\in A^g(\overline{M}_{g,n}(BG)).
\]
\end{definition}
\begin{remark}
The cycle $DR_g^G(A)$ is supported on the component of $\overline{M}_{g,n}(BG)$ parametrizing stable maps with orbifold structures at marked points given by $\{\mathfrak{c}_{0i}\}\cup\{\mathfrak{c}^{-1}_{\infty j}\}\cup I$.
\end{remark}

\section{Total Chern class}
\subsection{General case}\label{subsec:chern_gen}

Let $H$ be a finite group and $V=\mathbb C$ a one-dimensional $H$-representation. The representation $H\rightarrow GL(V)=\mathbb C^*$ maps a conjugacy class $\mathfrak c$ to $\exp\left(2\pi\sqrt{-1}\frac{a_{\mathfrak c}(V)}{r(\mathfrak c)}\right)$, where $a_{\mathfrak c}(V)\in\{0,\ldots, r(\mathfrak c)-1\}$.

We write $\on{Conj}(H)$ for the set of conjugacy classes of $H$. The inertia stack $IBH$ is decomposed as 
\[
IBH=\coprod_{\mathfrak c=(h)\in \on{Conj}(H)} BC_H(h)
\]
where $C_H(h)\subseteq H$ is the centralizer of $h\in H$.

We write $\overline{M}_{g,n}(BH)$ for the moduli stack of stable $n$-pointed genus $g$ maps to $BH$. For $1\leq i \leq n$, there is the $i$-th evaluation map
\[
\on{ev}_i: \overline{M}_{g,n}(BH)\longrightarrow IBH.
\]
Pick $\mathfrak c_1,\ldots, \mathfrak c_n \in \on{Conj}(H)$, let 
\[
\overline{M}_{g,n}(BH; \mathfrak c_1,\ldots,\mathfrak c_n):=\bigcap_{i=1}^n\on{ev}_i^{-1}(BC_H(h_i)),
\]
where $\mathfrak c_i=(h_i)$. Denote the universal family as follows:
\begin{displaymath}
    \xymatrix{ \mathcal C \ar[r]^f\ar[d]^\pi & BH\\
   \overline{M}_{g,n}(BH; \mathfrak c_1,\ldots,\mathfrak c_n)&}
\end{displaymath}

Consider the virtual bundle $$V_{g,n}:=\mathbf{R}\pi_*f^*V,$$ where $V$ is viewed as a line bundle on $BH$. The Chern character $ch(V_{g,n})$ was calculated in much greater generality in \cite{Tseng}, by using To\"en's Grothendieck-Riemann-Roch formula for stacks \cite{Toen}. Applied to the present situation, we find
\begin{align}\label{chern-character-formula}
ch(V_{g,n})=& \pi_*(ch(f^*V)Td^\vee(\bar{L}_{n+1}))\\
\notag & -\sum_{i=1}^n\sum_{m\geq 1}\frac{\on{ev}_i^*A_m}{m!}\bar{\psi}_i^{m-1}\\
\notag & +\frac 1 2 (\pi\circ \iota)_*\sum_{m\geq 2} \frac {1} {m!} r^2_{\text{node}}(\on{ev}^*_{\text{node}}A_m)\frac{\bar\psi_+^{m-1}+(-1)^m\bar\psi_-^{m-1}}{\bar\psi_+ +\bar\psi_-} 
\end{align}
The formula is explained and further processed as follows.
\begin{itemize}

\item $r_{\text{node}}$ is the order of the orbifold structure at the node.

\item $\on{ev}_{\text{node}}$ is the evaluation map at the node defined in \cite[Appendix B]{Tseng}.

\item $\bar\psi_+$ and $\bar\psi_-$ are the $\bar\psi$-classes associated the the branches of the node.

\item Since $\dim BH=0$, we have
\[
ch(f^*V)=ch_0(f^*V)=\rank V=1.
\]
\item By definition, the Todd class is 
\[
Td^\vee(\bar{L}_{n+1})=\frac{\bar\psi_{n+1}}{e^{\bar\psi_{n+1}}-1}=\sum_{r\geq 0}\frac{B_r}{r!}\bar\psi_{n+1}^r,
\]
where $B_r$'s are the Bernoulli numbers. Therefore,
\[
\pi_*(ch(f^*V)Td^\vee(\bar{L}_{n+1}))=\sum_{r\geq 0}\frac{B_r}{r!}\pi_*(\bar\psi_{n+1}^r)
\]
\item $A_m$ is defined in \cite[Definition 4.1.2]{Tseng}. We have $A_m\in H^*(IBH)$. For $\mathfrak c=(h)\in \on{Conj}(H)$, the component of $A_m$ in $H^0(BC_H(h))\subset H^*(IBH)$ is $B_m\left(\frac{a_{\mathfrak c}(V)}{r(\mathfrak c)}\right)$. Here $B_m(x)$ are Bernoulli polynomials, defined by
$$\frac{te^{tx}}{e^t-1}=\sum_{m\geq 0}\frac{B_m(x)}{m!}t^m.$$

\item The map $\iota: \mathcal Z_{\on{node}}\hookrightarrow \mathcal C$ is the inclusion of the locus of the nodes. The last term of the right hand side of \eqref{chern-character-formula} may be rewritten using the map
\[
B_{\on{node}}\stackrel{\mathfrak i}{\hookrightarrow} \overline{M}_{g,n}(BH;\mathfrak c_1,\ldots, \mathfrak c_n),
\]
whose image is the locus of nodal curves. The map $\mathfrak i$ exhibits $B_{\text{node}}$ as the universal gerbe at the node, and hence degree of $\mathfrak i$ is $\frac{1}{r_{\text{node}}}$.
\end{itemize}

Given the above, we can write $ch_m(V_{g,n})$, the degree-$2m$ component of $ch(V_{g,n})$, as
\begin{align*}
ch_m(V_{g,n})=& \frac{B_{m+1}}{(m+1)!}\pi_*(\bar\psi_{n+1}^{m+1})\\
& +\sum_{i=1}^n \frac{1}{(m+1)!}B_{m+1}\left(\frac{a_{\mathfrak c_i}(V)}{r(\mathfrak c_i)}\right)\bar\psi_i^m\\
& +\frac 12 \sum_{\mathfrak c\in \on{Conj}(H)} \frac{r(\mathfrak c)}{(m+1)!}B_{m+1}\left(\frac{a_{\mathfrak c}(V)}{r(\mathfrak c)}\right)\zeta_{\mathfrak{c}*}\left(\frac{\bar\psi_+^m-(-\bar\psi_-)^m}{\bar\psi_++\bar\psi_-}\right)
\end{align*}
where $\zeta_\mathfrak{c}: B_{\text{node},\mathfrak c}\rightarrow \overline{M}_{g,n}(BH;\mathfrak c_1,\ldots, \mathfrak c_n)$ is the universal gerbe at the node whose orbifold structure is given by $\mathfrak c$.

Using the formula 
\[
c(-E^\bullet)=\exp\left(\sum_{m\geq 1}(-1)^m (m-1)!ch_m(E^\bullet)\right), \quad E^\bullet \in D^b,
\]
we can derive a formula for $c(-V_{g,n})$. To write this down we need  more notation.

As in \cite{JPPZ}, the strata of $\overline{M}_{g,n}$ are indexed by stable graphs. The strata of $\overline{M}_{g,n}(BH;\mathfrak c_1,\ldots, \mathfrak c_n)$ are indexed by stable graphs together with choices of conjugacy classes of $H$ describing orbifold structures.

 Let $G_{g,n}$ be the set of stable graphs of genus $g$ with $n$ legs. Following \cite{JPPZ}, a stable graph is denoted by
\[
\Gamma=(\mathrm V,\mathrm H,\mathrm L, g:\mathrm V\rightarrow \mathbb Z_{\geq 0}, v:\mathrm H\rightarrow V, \iota: \mathrm H\rightarrow \mathrm H)\in G_{g,n}.
\]
 Properties in \cite[Section 0.3.2]{JPPZ} are required for $\Gamma$.

\begin{remark}
The set of legs $\mathrm L(\Gamma)$ corresponds to the set of markings. The set of half edges $\mathrm H(\Gamma)$ corresponds to the union of the set of a side of an edge and the set of legs. Each half edge is labelled with a vertex $v\in \mathrm V(\Gamma)$. Each vertex $v\in \mathrm V(\Gamma)$ is labelled with a nonnegative integer $g(v)$, called the genus.
\end{remark}
\begin{definition}
We define $\chi_{\Gamma,H}$ to be the set of maps

\[
\chi:\mathrm H(\Gamma)\rightarrow \on{Conj}(H)
\]
such that,
\begin{itemize}
\item $\chi$ maps the $i$-th leg $\mathrm h_i$ to $\mathfrak c_i$, $1\leq i \leq n$;

\item for a vertex $v\in \mathrm V(\Gamma)$, there exists $(\alpha_j), (\beta_j)\in \on{Conj}(H)$, for $1\leq j \leq g(v)$, and $k_\mathrm h\in \chi(\mathrm h)$, for $\mathrm h\in v$, such that
\[
\prod_{\mathrm h\in v}k_\mathrm h=\prod_{j=1}^{g(v)}[\alpha_j,\beta_j];
\]
\item for an edge $e=(\mathrm h,\mathrm h^\prime)\in \mathrm E(\Gamma)$, there exists $k\in \chi(\mathrm h)$, $k^\prime \in \chi(\mathrm h^\prime)$, such that 
\[
kk^\prime=\on{Id}\in H.
\]
\end{itemize}
\end{definition}
For each $\Gamma\in G_{g,n}$ and $\chi \in \chi_{\Gamma, H}$, there is a component $\overline{M}_{\Gamma,\chi}\subset B_{\text{node}}$ parametrizing maps with nodal domains of topological types given by $\Gamma$ and orbifold structures given by $\chi$. Let 
\[
\zeta_{\Gamma,\chi}:\overline{M}_{\Gamma,\chi}\longrightarrow \overline{M}_{g,n}(BH;\mathfrak c_1,\ldots, \mathfrak c_n)
\]
be the restriction of $\mathfrak i$ to this component. Then $c(-V_{g,n})$ is
\begin{equation}
\begin{split}
& \sum_{\Gamma\in G_{g,n}}\sum_{\chi\in \chi_{\Gamma,H}}\frac{1}{|\Aut(\Gamma)|}\zeta_{\Gamma,\chi*}\left[
\prod_{v\in V(\Gamma)}\exp\left(-\sum_{m\geq 1}(-1)^{m-1}\frac{B_{m+1}}{m(m+1)}\kappa_m(v)\right)\times\right.\\
& \times \prod_{i=1}^n\exp\left(\sum_{m\geq1}(-1)^{m-1}\frac{1}{m(m+1)}B_{m+1}\left(\frac{a_{\mathfrak c_i}(V)}{r(\mathfrak c_i)}\right)\bar\psi_{\mathrm h_i}^m\right)\times\\
& \left. \times \prod_{\substack{e\in \mathrm E(\Gamma)\\ e=(\mathrm h_+,\mathrm h_-)}}r(\chi(\mathrm h_+))\frac{1}{\bar\psi_{\mathrm h_+}+\bar\psi_{\mathrm h_-}}
\left(1-\exp\left(\sum_{m\geq 1}\frac{(-1)^{m-1}}{m(m+1)}B_{m+1}\left(\frac{a_{\chi(\mathrm h_+)}(V)}{r(\chi(\mathrm h_+))}\right)\left(\bar\psi_{\mathrm h_+}^m-(-\bar\psi_{\mathrm h_-})^m\right)\right)\right)\right].
\end{split}
\end{equation}
\begin{remark}
\hfill
\begin{enumerate}
\item
For a half-edge $h$, $\bar{\psi}_h$ denotes the descendant at the marked point/node corresponding to $h$. 

\item
For a vertex $v$, let $\overline{M}_v(BH)$ be the moduli space of stable maps to $BH$ described by $v$ and let $\pi_v: \mathcal{C}_v\to \overline{M}_v(BH)$ be the universal curve. Write $\bar{\psi}_v\in A^1(\mathcal{C}_v)$ for the descendant corresponding to the additional marked point. Then define $\kappa_m(v):= \pi_{v*}(\bar\psi_{v}^{m+1})$.
\end{enumerate}
\end{remark}

\subsection{Cyclic extensions}\label{subsec:gp_ext}
Let $r\in \mathbb Z_{>0}$, the $r$-th power map 
\[
\mathbb C^*\longrightarrow \mathbb C^*, \quad z \longmapsto z^r
\]
gives the map
\[
\mu_{ar}\longrightarrow \mu_a.
\]
The kernel of the map is $\mu_r$. Hence this gives the exact sequence
\[
1\longrightarrow \mu_r \stackrel{g}{\longrightarrow} \mu_{ar}\stackrel{f}{\longrightarrow} \mu_a \longrightarrow 1,
\]
where 
\[
g\left(\exp\left(\frac{2\pi\sqrt{-1}l}{r}\right)\right)=\exp\left(\frac{2\pi\sqrt{-1}la}{ra}\right), \quad 0\leq l \leq r-1,
\]
and
\[
f\left(\exp\left(\frac{2\pi\sqrt{-1}k}{ar}\right)\right)=\exp\left(\frac{2\pi\sqrt{-1}k}{a}\right), \quad 0\leq k \leq ar-1.
\]
There is a unique finite group $G(r)$ which fits into the following diagram with exact rows and columns:
\begin{equation}\label{diag:ext}
    \xymatrix{              &                                 & 1                 \ar[d]                 & 1                 \ar[d]          &   \\
            &                                 & \mu_r\ar@{=}[r]                 \ar[d]                 & \mu_r                \ar[d]          &   \\
    1 \ar[r] & K \ar[r]\ar@{=}[d] & G(r)                 \ar[r]^{\alpha}\ar[d]^{\beta}        & \mu_{ar}         \ar[r]\ar[d] & 1 \\
    1 \ar[r] & K \ar[r]                 & G             \ar[r]^{\varphi}\ar[d]        & \mu_a        \ar[r]\ar[d] & 1 \\
                &                                 & 1                                           & 1                                    &}
\end{equation}
Geometrically, the map $\mu_{ar}\rightarrow \mu_a$ gives a $\mu_r$-gerbe over $B\mu_a$,
\[
B\mu_{ar}\longrightarrow B\mu_a.
\]
The map $\varphi: G\rightarrow \mu_a$ gives a map
\[
BG\longrightarrow B\mu_a.
\]
Pulling back the $\mu_r$-gerbe to $BG$ using this map, we obtain the gerbe
\begin{equation}\label{eqn:bgr_bg}
BG(r)\longrightarrow BG.
\end{equation}
Moreover, when viewing the representation $L$ as a line bundle on $BG$, $BG(r)$ is the gerbe of $r$-th roots of $L\rightarrow BG$. The homomorphism 
\[
G(r)\longrightarrow \mu_{ar}\subset \mathbb C^*
\] 
is a one-dimensional representation of $G(r)$ which corresponds to the universal $r$-th root of $L$ on $BG(r)$. We denote this $r$-th root by
\[
L^{1/r}\longrightarrow BG(r).
\] 

Let $\mathfrak c\in \on{Conj}(G)$. Then $\varphi(\mathfrak c)\in \mu_a$ is a single number. The inverse image of $\varphi(\mathfrak c)$ under the $r$-th power map $\mu_{ar}\rightarrow \mu_a$ has size $r$. The inverse image $\beta^{-1}(\mathfrak c)\subset G(r)$ can be partitioned into conjugacy classes of $G(r)$. Moreover, $\alpha$ maps these conjugacy classes to the set of inverse images of $\varphi(\mathfrak c)$, which has size $r$. So there are at least $r$ conjugacy classes in $\beta^{-1}(\mathfrak c)$. By the counting result \cite[Example 3.4]{TT1}, there are at most $r$ conjugacy classes. Therefore, there are exactly $r$ conjugacy classes of $G(r)$ that map to $\mathfrak c$ and they are determined by their images under $G(r)\stackrel{\alpha}{\rightarrow} \mu_{ar}$. 

A canonical splitting of 
\[
1\longrightarrow \mu_r \longrightarrow \mu_{ar}\longrightarrow \mu_a\longrightarrow 1
\]
is given by 
\begin{equation}\label{eqn:splitting}
\mu_a\longrightarrow \mu_{ar}, \quad g\mapsto \exp \left(\frac{2\pi\sqrt {-1}\on{age}_g(L)}{r} \right).
\end{equation}
Using this, for $g\in\mu_a$, we may identify the inverse image of $g$ under $\mu_{ar}\rightarrow \mu_a$ as
\[
\left\{\exp\left.\left(2\pi\sqrt {-1}\left(\frac{\on{age}_g(L)}{r}+\frac e r \right)\right)\right| 0\leq e \leq r-1\right\}
\]
and hence with 
\[
\mu_r=\left\{\exp\left.\left(\frac{2\pi\sqrt{-1}e}{r}\right)\right|0\leq e\leq r-1 \right\}.
\]
In summary, given $\mathfrak c\in\on{Conj}(G)$, to specify the lifting $\tilde {\mathfrak c} \in \on{Conj}(G(r))$ such that $\beta(\tilde{\mathfrak c})=\mathfrak c$ is equivalent to specifying $e\in \{0,\ldots, r-1\}$. 

Moreover, given $\mathfrak c_1,\ldots, \mathfrak c_n \in \on{Conj}(G)$ satisfying monodromy condition in genus $g$, selecting $\tilde{\mathfrak c}_1,\ldots, \tilde{\mathfrak c}_n \in \on{Conj}(G(r))$ with $\beta(\tilde{\mathfrak c}_i)=\mathfrak c_i$ satisfying monodromy condition in genus $g$ is equivalent to selecting $e_1,\ldots, e_n \in \{0,\ldots,r-1\}$ such that 
\begin{equation}\label{eqn:monod_r}
\sum_{i=1}^n e_i \equiv -\sum_{i=1}^n\text{age}_{\mathfrak{c}_i}(L) \mod r.
\end{equation}
This can be deduced from the lifting analysis in \cite[Section 5]{TT}. We can also argue more directly as follows. Since $\mathfrak c_1,\ldots, \mathfrak c_n$ satisfy monodromy condition in genus $g$, there exists a stable map $f: \mathcal{C}\to BG$ with $\mathcal{C}$ smooth of genus $g$ and $\mathcal{C}$ has orbifold points described by $\mathfrak c_1,\ldots, \mathfrak c_n$. Calculating $\chi(\mathcal{C}, f^*L)$ by Riemann-Roch, we see that $\sum_{i=1}^n\text{age}_{\mathfrak{c}_i}(L)\in \mathbb{Z}$. Similarly, having the required $\tilde{\mathfrak c}_1,\ldots, \tilde{\mathfrak c}_n$ implies the existence of a stable map $\tilde{f}: \tilde{\mathcal{C}}\to BG(r)$ with $\tilde{\mathcal{C}}$ smooth of genus $g$ and $\tilde{\mathcal{C}}$ has orbifold points described by $\tilde{\mathfrak{c}}_1,\ldots, \tilde{\mathfrak{c}}_n$. Calculating $\chi(\tilde{\mathcal{C}}, \tilde{f}^*L^{1/r})$ by Riemann-Roch, we see that $\sum_{i=1}^n\text{age}_{\tilde{\mathfrak{c}}_i}(L^{1/r})\in \mathbb{Z}$. Equation (\ref{eqn:monod_r}) follows because by construction $\text{age}_{\tilde{\mathfrak{c}}_i}(L^{1/r})=(\text{age}_{\mathfrak{c}_i}(L)+e_i)/r$. This shows that equation (\ref{eqn:monod_r}) is necessary. That (\ref{eqn:monod_r}) is also sufficient can be seen by a direct calculation using the description of $G(r)$ as a set $G\times \mu_r$ endowed with the multiplication defined using the splitting (\ref{eqn:splitting}), as in \cite[Section 3]{TT0}. We omit the details.

The above discussion allows us to split a sum over $\chi_{\Gamma,G(r)}$ as a double sum over $\chi_{\Gamma,G}$ and the set $W_{\Gamma,\chi,r}$ defined as follows. 

\begin{definition}\label{Weightings}
A weighting mod $r$ associated to a stable graph $\Gamma$ and a map $\chi\in \chi_{\Gamma, G}$ is a function 
\[
w:\mathrm H(\Gamma)\longrightarrow \{0,\ldots, r-1\}
\]
such that
\begin{enumerate}
\item 
For legs $\mathrm  h_1,\ldots,\mathrm h_n$, modulo $r$, $w(\mathrm h_i)$ is either 
\begin{itemize}
    \item 
     $\frac{c_{0i}}{f_{0i}}-\on{age}_{\mathfrak c_{0i}}(L)$ for $\mathfrak{c}_{0i}$, or
    
    \item
    $\frac{c_{\infty j}}{f_{\infty j}}-\on{age}_{\mathfrak c_{\infty j}}(L)$ for $\mathfrak{c}_{\infty j}$, or
    
    \item
    $0$ for $\mathfrak{c}_k$.
\end{itemize}

\item 
For $e=(\mathrm h_+,\mathrm h_-)\in \mathrm E(\Gamma)$, if $\on{age}_{\chi(\mathrm h_+)}(L)=0$, then $w(\mathrm h_+)+w(\mathrm h_-)\equiv 0 \mod r$. If $\on{age}_{\chi(\mathrm h_+)}(L)\neq 0$, then $w(\mathrm h_+)+w(\mathrm h_-)\equiv -1 \mod r$.

\item For $v\in \mathrm V(\Gamma)$, $\sum_{\mathrm h\in v}w(\mathrm h)\equiv A(v,\chi) \mod r$, where $A(v,\chi):=-\sum_{\mathrm h\in v}\on{age}_{\chi(\mathrm h)}(L)$. 
\end{enumerate}
We write $W_{\Gamma,\chi,r}$ for the set of weightings mod $r$ associated to $\Gamma$ and $\chi$.
\end{definition}
\begin{remark}
\hfill
\begin{enumerate}
\item
For $e=(\mathrm h_+,\mathrm h_-)\in \mathrm E(\Gamma)$, the conditions on $w(h_\pm)$ ensure that $$(\text{age}_{\chi(h_-)}(L)+w(h_-))/r=1-(\text{age}_{\chi(h_+)}(L)+w(h_+))/r.$$
\item
For $v\in \mathrm V(\Gamma)$, We have $A(v,\chi)\in \mathbb Z$ by applying Riemann-Roch to $\chi(f^*L)$, where $f: C\rightarrow BG$ is a stable map with $ C$ smooth of genus $g(v)$ and orbifold marked points described by $\{\chi(\mathrm h)|\mathrm h\in v\}$.
\end{enumerate}
\end{remark}

\subsection{Total Chern class on moduli spaces of stable maps to $BG(r)$} 
We begin with the following notation. 
\begin{definition}[Liftings]\label{defn:lift}
For $\{\mathfrak c_{0i}\}_i, \{\mathfrak c_{\infty j}\}_j \subset \on{Conj}(G)$, we select liftings 
\[
\{\tilde{\mathfrak c}_{0i}\}_i, \{\tilde{\mathfrak c}_{\infty j}\}_j \subset \on{Conj}(G(r))
\]
by 
\[
\alpha(\tilde{\mathfrak c}_{0i})=\exp\left(\frac{2\pi \sqrt{-1}\frac{c_{0i}}{f_{0i}}}{r}\right)\in\mu_{ar},
\]
\[
\alpha(\tilde{\mathfrak c}_{\infty j})=\exp\left(\frac{2\pi \sqrt{-1}(-\frac{c_{\infty j}}{f_{\infty j}})}{r}\right)\in\mu_{ar}.
\]
The lifts of $\mathfrak c_1,\ldots, \mathfrak c_k \in \on{Conj}(K)\subset \on{Conj}(G)$ are chosen to be themselves, viewed via $\on{Conj}(K)\subset \on{Conj}(G(r))$.
\end{definition}

Let $\overline{M}_{g,\tilde{\mu}_0+\tilde{\mu}_\infty+I}(BG(r))$ be the moduli space of stable maps to $BG(r)$ of genus $g$ whose marked points have orbifold structures given by
\[
\{\tilde{\mathfrak c}_{0i}\}\cup \{\tilde{\mathfrak c}_{\infty j}^{-1}\}\cup\{\mathfrak c_1,\ldots,\mathfrak c_k\}.
\]
Let $\overline{M}_{g,\mu_+\mu_\infty+I}(BG)$ be similarly defined. There is a natural map 
\[
\epsilon: \overline{M}_{g,\tilde{\mu}_0+\tilde{\mu}_\infty+I}(BG(r))\longrightarrow \overline{M}_{g,\mu_0+\mu_\infty+I}(BG).
\]
Strata of $\overline{M}_{g,\mu_0+\mu_\infty+I}(BG)$ are indexed by pairs $\Gamma\in G_{g,n}$ and $\chi \in \chi_{\Gamma,G}$. Let $\zeta_{\Gamma,\chi}$ be the map from this stratum to $\overline{M}_{g,\mu_0+\mu_\infty+I}(BG)$.
Strata of $\overline{M}_{g,\tilde{\mu}_0+\tilde{\mu}_\infty+I}(BG(r))$ are indexed by $\Gamma,\chi$, and $w\in W_{\Gamma,\chi,r}$. Let $\zeta_{\Gamma,\chi,w}$ be the natural map from the stratum to  $\overline{M}_{g,\tilde{\mu}_0+\tilde{\mu}_\infty+I}(BG(r))$.

Applying the results of Section \ref{subsec:chern_gen}, we obtain the following formula for $c(-L_{g,n}^{1/r})$ on $\overline{M}_{g,\tilde{\mu}_0+\tilde{\mu}_\infty+I}(BG(r))$:
\begin{align*}
& \sum_{\Gamma\in G_{g,n}}\sum_{\chi\in \chi_{\Gamma,G}}\sum_{w\in W_{\Gamma,\chi,r}}\frac{1}{|\Aut(\Gamma)|}\zeta_{\Gamma,\chi,w*}\left[
\prod_{v\in \mathrm V(\Gamma)}\exp\left(-\sum_{m\geq 1}(-1)^{m-1}\frac{B_{m+1}}{m(m+1)}\kappa_m(v)\right)\times\right.\\
& \times \prod_{i}\exp\left(\sum_{m\geq1}(-1)^{m-1}\frac{1}{m(m+1)}B_{m+1}\left(\frac{\frac{c_{0i}}{f_{0i}}}{r}\right)\bar\psi_{i}^m\right)\times\\
& \times \prod_{j}\exp\left(\sum_{m\geq1}(-1)^{m-1}\frac{1}{m(m+1)}B_{m+1}\left(1-\frac{\frac{c_{\infty j}}{f_{\infty j}}}{r}\right)\bar\psi_{j}^m\right)\times\\
& \times \prod_{l=1}^k\exp\left(\sum_{m\geq1}(-1)^{m-1}\frac{1}{m(m+1)}B_{m+1}\bar\psi_{l}^m\right)\\
& \left. \times \prod_{\substack{e\in \mathrm E(\Gamma)\\ e=(\mathrm h_+,\mathrm h_-)}}\frac{r(\chi(\mathrm h_+))r}{\bar\psi_{\mathrm h_+}+\bar\psi_{\mathrm h_-}}
\left(1-\exp\left(\sum_{m\geq 1}\frac{(-1)^{m-1}}{m(m+1)}B_{m+1}\left(\frac{\text{age}_{\chi(\mathrm h_+)}(L)}{r}+\frac{w(\mathrm h_+)}{r}\right)\left(\bar\psi_{\mathrm h_+}^m-(-\bar\psi_{\mathrm h_-})^m\right)\right)\right)\right].
\end{align*}
By the calculation of \cite[Section 5]{TT}, the degree of $\epsilon$ on strata indexed by $\Gamma$ is $r^{\sum_{v\in V(\Gamma)}(2g(v)-1)}$. This yields the following formula for $\epsilon_*c(-L_{g,n}^{1/r})$:
\begin{align*}
& \sum_{\Gamma\in G_{g,n}}\sum_{\chi\in \chi_{\Gamma,G}}\sum_{w\in W_{\Gamma,\chi,r}}\frac{r^{2g-1-h^1(\Gamma)}}{|\Aut(\Gamma)|}\zeta_{\Gamma,\chi*}\left[
\prod_{v\in \mathrm V(\Gamma)}\exp\left(-\sum_{m\geq 1}(-1)^{m-1}\frac{B_{m+1}}{m(m+1)}\kappa_m(v)\right)\times\right.\\
& \times \prod_{i}\exp\left(\sum_{m\geq1}(-1)^{m-1}\frac{1}{m(m+1)}B_{m+1}\left(\frac{\frac{c_{0i}}{f_{0i}}}{r}\right)\bar\psi_{i}^m\right)\times\\
& \times \prod_{j}\exp\left(\sum_{m\geq1}(-1)^{m-1}\frac{1}{m(m+1)}B_{m+1}\left(1-\frac{\frac{c_{\infty j}}{f_{\infty j}}}{r}\right)\bar\psi_{j}^m\right)\times\\
& \times \prod_{l=1}^k\exp\left(\sum_{m\geq1}(-1)^{m-1}\frac{1}{m(m+1)}B_{m+1}\bar\psi_{l}^m\right)\\
& \left. \times \prod_{\substack{e\in \mathrm E(\Gamma)\\ e=(\mathrm h_+,\mathrm h_-)}}\frac{r(\chi(\mathrm h_+))}{\bar\psi_{\mathrm h_+}+\bar\psi_{\mathrm h_-}}
\left(1-\exp\left(\sum_{m\geq 1}\frac{(-1)^{m-1}}{m(m+1)}B_{m+1}\left(\frac{\text{age}_{\chi(\mathrm h_+)}(L)}{r}+\frac{w(\mathrm h_+)}{r}\right)\left(\bar\psi_{\mathrm h_+}^m-(-\bar\psi_{\mathrm h_-})^m\right)\right)\right)\right].
\end{align*}
Note that we have
\begin{itemize}
\item $\frac{\text{age}_{\chi(\mathrm h_+)}}{r}+\frac{w(\mathrm h_+)}{r}=1-\frac{\text{age}_{\chi(\mathrm h_-)}(L)}{r}-\frac{w(\mathrm h_-)}{r}$ if $\text{age}_{\chi(\mathrm h_{\pm})}(L)\neq 0$.
\item $\frac{w(\mathrm h_+)}{r}=1-\frac{w(\mathrm h_-)}{r}$, if $\text{age}_{\chi(\mathrm h_{\pm})}(L)=0$. 
\end{itemize}
Bernoulli polynomials satisfy the following property $$B_m(x+y)=\sum_{l=0}^m\binom{m}{k}B_k(x)y^{m-k}.$$ 
This implies that terms of $\epsilon_*c(-L_{g,n}^{1/r})$ depend polynomially on $\{w(\mathrm h)|\mathrm h\in \mathrm H(\Gamma)\}$.
The proof of \cite[Proposition 3'']{JPPZ} may be modified to show that the polynomiality result remains valid for sums over $W_{\Gamma,\chi,r}$. Therefore we may apply the arguments of \cite[Proposition 5]{JPPZ} to conclude the following.
\begin{proposition}\label{prop:poly}
There exists a polynomial in $r$ which coincides with the cycle class $r^{2d-2g+1}\epsilon_*c_d(-L_{g,n}^{1/r})$ for $r\gg1$ and prime.
\end{proposition}

\begin{remark}
The orbifold structure at $\mathrm h_{\pm}$ has order $r(\chi(\mathrm h_{\pm}))r$ when $r\gg 1$ are primes. For our purpose this suffices.
\end{remark}
\subsection{A formula for stacky double ramification cycle}
\begin{theorem}\label{dr-cycle-theorem}
Given a finite group $G$ and double ramification data $A=\{\mu_0,\mu_\infty,I\}$, the stacky double ramification cycle $DR_g^G(A)$ is the constant term in $r$ of the cycle class
\[
a^{l(\mu_\infty)-l(\mu_0)}r\cdot\epsilon_* c_g(-L_{g,n}^{1/r})\in A^g(\overline{M}_{g,n}(BG)),
\]
for $r$ sufficiently large. In other words,
\[
DR_g^G(A)=a^{l(\mu_\infty)-l(\mu_0)}\on{Coeff}_{r^0}[r\cdot\epsilon_* c_g(-L_{g,n}^{1/r})]\in A^g(\overline{M}_{g,n}(BG)).
\]
\end{theorem}

We denote by $P_{g}^{G,d,r}(A)\in A^d(\overline{M}_{g,n}(BG))$ the degree $d$ component of the class
\begin{align*}
&\sum_{\Gamma\in G_{g,n}}\sum_{\chi\in \chi_{\Gamma,G}}\sum_{w\in W_{\Gamma,\chi,r}}
\frac{1}{|\Aut(\Gamma)|}\frac{1}{r^{h^1(\Gamma)}}\zeta_{\Gamma, \chi *}\left[ \prod_{i}\exp((\frac{c_{0i}}{f_{0i}})^2\bar\psi_{i})\prod_{j}\exp((\frac{c_{\infty j}}{f_{\infty j}})^2\bar\psi_{j})\cdot\right.\\
&\left.\prod_{\substack{e\in \mathrm E(\Gamma)\\ e=(\mathrm h_+,\mathrm h_-)}}r(\chi(\mathrm h_+))
\frac{1-\exp(-(\text{age}_{\chi(\mathrm h_+)}(L)+w(\mathrm h_+))(\text{age}_{\chi(\mathrm h_-)}(L)+w(\mathrm h_-))(\bar\psi_{\mathrm h_+}+\bar\psi_{\mathrm h_-}))}{\bar\psi_{\mathrm h_+}+\bar\psi_{\mathrm h_-}}\right].
\end{align*}
When the finite group $G$ is trivial, $P_g^{G,d,r}$ reduces to Pixton's polynomial in \cite{JPPZ}. 
Arguing as in the proof of \cite[Proposition 5]{JPPZ}, we see that $r^{2d-2g+1}\epsilon_* c_d(-L_{g,n}^{1/r})$ and $2^{-d}P_{g}^{G,d,r}(A)$ have the same constant term. Then the following corollary is a result of Theorem \ref{dr-cycle-theorem}.

\begin{corollary}
The stacky double ramification cycle $DR^G_g(A)$  is the constant term in $r$ of $a^{l(\mu_\infty)-l(\mu_0)}2^{-g}P_{g}^{G,g,r}(A)\in A^g(\overline{M}_{g,n}(BG))$, for $r$ sufficiently large.
\end{corollary}

\section{Localization analysis}
In this section, we give a proof of Theorem \ref{dr-cycle-theorem} by virtual localization on the moduli space of stable relative maps to the target obtained from $[\P^1/G]$ by a root construction.

\subsection{Set-up}
Let $[\P^1/G]_{r,1}$ be the stack of $r$-th roots of $[\P^1/G]$ along the divisor $[0/G]$. By construction, there is a map $[\P^1/G]_{r,1}\to [\P^1/G]$. Over the divisor $[0/G]\simeq BG$, this map is the $\mu_r$-gerbe $BG(r)\to BG$ studied in Section \ref{subsec:gp_ext}.

Let $\mu_0, \mu_\infty, I$ be as in Definition \ref{defn:rel_data}. Let $\tilde{\mu}_0=\{(c_{0i},f_{0i},\tilde{\mathfrak c}_{0i})\}_i$, where $\tilde{\mathfrak{c}}_{0i}$ are given in Definition \ref{defn:lift}. Let
\[
\overline{M}_{g,I,\mu_0}([\P^1/G]_{r,1},\mu_\infty)
\]
be the moduli space of stable relative maps to the pair $([\P^1/G]_{r,1},[\infty/G])$. The moduli space parametrizes connected, semistable, twisted curves $C$ of genus $g$ with non-relative marked points together with a map
\[
f:C\rightarrow P
\]
where $P$ is an expansion of $[\P^1/G]_{r,1}$ over $[\infty/G]$ such that
\begin{enumerate}
\item
orbifold structures at the non-relative marked points are described by $\tilde{\mu}_0$ and $I$;
\item
relative conditions over $[\infty/G]$ are described by $\mu_\infty$.
\item
The map $f$ satisfies the ramification matching condition over the internal nodes of the destabilization $P$.
\end{enumerate}

By \cite{af},  $\overline{M}_{g,I,\mu_0}([\P^1/G]_{r,1},\mu_\infty)$ has a perfect obstruction theory and its virtual fundamental class has complex dimension
\[
\text{vdim}_\C [\overline{M}_{g,I,\mu_0}([\P^1/G]_{r,1},\mu_\infty)]^{\text{vir}}=2g-2+n+\frac{|\mu_\infty|}{r}-\sum\limits_{i=1}^{l(\mu_0)}\text{age}_{\tilde{\mathfrak{c}}_{0i}}(L^{1/r}),\]
where $n=l(\mu_0)+l(\mu_\infty)+\#I$ and $|\mu_\infty|:=\sum_j \frac{c_{\infty j}}{f_{\infty j}}$.

For $r\gg1$, we have $\text{age}_{\tilde{\mathfrak{c}}_{0i}}(L^{1/r})=\frac{c_{0i}}{rf_{0i}}$. In this case the virtual dimension is $2g-2+n$.

In what follows, we assume that $r$ is large and is a prime number.

\subsection{Fixed loci}
The standard $\C^*$-action on $\P^1$ is given by $$\xi\cdot [z_0, z_1]:=[z_0, \xi z_1], \quad \xi\in \C^*, \, [z_0, z_1]\in \P^1.$$
This induces $\C^*$-actions on $[\P^1/G]$, $[\P^1/G]_{r,1}$, and $\overline{M}_{g,I,\mu_0}([\P^1/G]_{r,1},\mu_\infty)$. The $\C^*$-fixed loci of $\overline{M}_{g,I,\mu_0}([\P^1/G]_{r,1},\mu_\infty)$ are labeled by decorated graphs $\Gamma$. 

\begin{notation}
A decorated graph $\Gamma$ is defined as follows:
\begin{enumerate}
\item (Graph data)
\begin{itemize}

\item $V(\Gamma)$: the set of vertices of $\Gamma$;

\item $E(\Gamma)$: the set of edges of $\Gamma$;

\item $F(\Gamma)$: the set of flags of $\Gamma$, defined to be
\[
F(\Gamma)=\{(e,v)\in E(\Gamma)\times V(\Gamma)| v\in e\};
\]
\item
$L(\Gamma)$: the set of legs;
\end{itemize}
\item (Decoration data)

\begin{itemize}
\item each vertex $v\in V(\Gamma)$ is assigned a genus $g(v)$, a label of either $[0/G(r)]$ or $[\infty/G]$, and a group 
$$G_v:=
\begin{cases}
G(r)\quad \text{if } v \text{ is over } [0/G(r)],\\
G\quad \text{if } v \text{ is over } [\infty/G].
\end{cases}
$$

\item each edge $e\in E(\Gamma)$ is labelled with a conjugacy class $(k_e)\subset G_e:=K$ and a positive integer $d_e$, called the degree;

\item each flag $(e,v)$ is labelled with a conjugacy class $(k_{(e,v)})\subset G_v$;

\item a map $s:L(\Gamma) \rightarrow V(\Gamma)$ that assigns legs to vertices of $\Gamma$;

\item legs are labelled with markings in $\mu_0\cup I\cup \mu_\infty$. Namely $j\in L(\Gamma)$ is labelled with a conjugacy class $(k_j)\subset G_v$ where
$$\begin{cases}
(k_j)\in \{\tilde{\mathfrak{c}}_{0i}\}_i\cup I\quad \text{if } v \text{ is over } [0/G(r)];\\
(k_j)\in \{\mathfrak{c}_{\infty j}\}_j\cup I \quad \text{if } v \text{ is over }[\infty/G].
\end{cases}$$

\end{itemize}

\end{enumerate}
The data above satisfy certain compatibility conditions. We omit them as they do not enter our analysis.
\end{notation}

A vertex $v\in V(\Gamma)$ over $[0/G(r)]$ corresponds to a contracted component mapping to $[0/G(r)]$ given by an element of the moduli space $$\overline{M}_v:=\overline{M}_{g(v),I(v),\mu_0(v)}(BG(r))$$ 
of genus $g(v)$ stable maps to $BG(r)$ such that orbifold structures at marked points are given by corresponding entries of $\tilde{\mu}_0$ and $(k_{(e,v)})^{-1}$ for flags attached to $v$. The dimension of $\overline{M}_{g(v),I(v),\mu_0(v)}(BG(r))$ is $3g(v)-3+\#I(v)+l(\mu_0(v))+|E(v)|$.

The discussion on fixed stable maps over $[\infty/G]\in [\P^1/G]_{r,1}$ is similar to that in \cite[Section 2.3]{JPPZ}, we omit the details.

Let $\overline{M}_\infty^{\sim}$ be the moduli space of stable maps to rubber. Its virtual class $[\overline{M}_\infty^{\sim}]^\text{vir}$ has complex dimension $2g(\infty)-3+n(\infty)$, where $g(\infty)$ is the domain genus and $n(\infty)=\# I(\infty)+l(\mu_\infty)+|E(\Gamma)|$ is the total number of markings and incidence edges.

We write $V_0^S(\Gamma)$ for the set of stable vertices of $\Gamma$ over $[0/G(r)]$. If the target degenerates, define
\[
\overline{M}_{\Gamma}=\prod_{v\in V_0^S(\Gamma)}\overline{M}_{g(v),I(v),\mu_0(v)}(BG(r))\times \overline{M}_\infty^{\sim}. 
\]
If the target does not degenerate, define
\[
\overline{M}_{\Gamma}=\prod_{v\in V_0^S(\Gamma)}\overline{M}_{g(v),I(v),\mu_0(v)}(BG(r)).
\]
The fixed locus corresponding to $\Gamma$ is isomorphic to $\overline{M}_\Gamma$ quotiented by the automorphism group of $\Gamma$ and the product of cyclic groups associated to the Galois covers of the edges. There is a natural map $\iota: \overline{M}_\Gamma\to \overline{M}_{g,I,\mu_0}([\P^1/G]_{r,1},\mu_\infty)$.

Assuming $r\gg1$, we may argue as in \cite[Lemma 6]{JPPZ} to conclude that there are only two types of unstable vertices:
\begin{itemize}
\item $v$ is mapped to $[0/G]$, $g(v)=0$, $v$ carries one marking and one incident edge;

\item $v$ is mapped to $[\infty/G]$, $g(v)=0$, $v$ carries one marking and one incident edge.
\end{itemize}

\subsection{Contributions to localization formula}
By convention, the $\C^*$-equivariant Chow ring of a point is identified with $\Q[t]$ where $t$ is the first Chern class of the standard representation.

Let $[f: C\to [\P^1/G]_{r,1}]\in \overline{M}_\Gamma$. The $\C^*$-equivariant Euler class of the virtual normal bundle in $\overline{M}_{g,I,\mu_0}([\P^1/G]_{r,1},\mu_\infty)$ to the $\mathbb C^*$-fixed locus indexed by $\Gamma$ can be described as
\[
e(N^{\on{vir}})^{-1}=\frac{e(H^1(C,f^*T_{[\P^1/G]_{r,1}}(-[\infty/G])))}{e(H^0(C,f^*T_{[\P^1/G]_{r,1}}(-[\infty/G])))}\left(\prod_i e(N_i)\right)^{-1}e(N_\infty)^{-1}.
\]

Let $V^S(\Gamma)$ be the set of stable vertices in $V(\Gamma)$. The set of stable flags is defined to be
\[
F^S(\Gamma)=\{(e,v)\in F(\Gamma)| v\in V^S(\Gamma)\}.
\]
We have
\begin{equation}\label{localization-equ}
[\overline{M}_{g,I,\mu_0}([\P^1/G]_{r,1},\mu_\infty)]^{\on{vir}}=\sum_\Gamma \frac{1}{|Aut(\Gamma)|}\frac{1}{\prod_{e\in E(\Gamma)}d_e|G_e|}\left(\prod_{(e,v)\in F^S(\Gamma)}\frac{|G_v|}{r_{(e,v)}}\right)\iota_*\left(\frac{[\overline{M}_{\Gamma}]^{\on{vir}}}{e(N^{\on{vir}})}\right)
\end{equation}
where $r_{(e,v)}$ is the order of $k_{(e,v)}\in G_v$.

The localization contributions are given as follows.
\begin{enumerate}
\item
Contributions to
\[
\frac{e(H^1(C,f^*T_{[\P^1/G]_{r,1}}(-[\infty/G])))}{e(H^0(C,f^*T_{[\P^1/G]_{r,1}}(-[\infty/G])))}.
\]

\begin{enumerate}
\item
For each stable vertex $v\in V(\Gamma)$ over $0$, the contribution is
\[
c^{\C^*}_\text{rk}((-L^{1/r}(v)))=\sum_{d\geq 0}c_d(-L_{g,n}^{1/r})\left(\frac{t}{r}\right)^{g(v)-1+|E(v)|-d}.
\]
Here  
the virtual rank $\text{rk}$ is $g(v)-1+|E(v)|$. This follows from Riemann-Roch together with the observation that because $r\gg 1$, the age terms in Riemann-Roch add up to $|E(v)|$.

\item
The two possible unstable vertices contribute  $1$.

\item
The edge contribution is trivial because $r\gg 1$.

\item
The contribution of a node $N$ over $[0/G(r)]$ is trivial.

\item
Nodes over $[\infty/G]$ contribute $1$.

\end{enumerate}

\item
Contributions to $\prod_i e(N_i)$.

The product $\prod_i e(N_i)$ is over all nodes over $[0/G(r)]$ formed by edges of $\Gamma$ attaching to vertices. If $N$ is such a node, then
\[
e(N)=\frac{t}{r_{(e,v)}d_e}-\frac{\psi_e}{r_{(e,v)}}.
\]

Hence, combining (i) and (ii), the contribution of this stable vertex $v$ is:
\begin{align*}
&\prod_{e\in E(v)}\frac{1}{\frac{t}{r_{(e,v)}d_e}-\frac{\psi_e}{r_{(e,v)}}}\sum_{d\geq 0}c_d(-L_{g,n}^{1/r})\left(\frac{t}{r}\right)^{g(v)-1+|E(v)|-d}.
\end{align*}

\item
Contributions to $e(N_\infty)$.

If the target degenerates, there is an additional factor
\[
\frac{1}{e(N_\infty)}=-\frac{\prod_{e\in E(\Gamma)}d_e r_{(e,v)}}{t+\psi_\infty}.
\]
\end{enumerate}
\subsection{Extraction}
The virtual class of the moduli space of rubber maps has non-equivariant limit, and $\C^*$ acts trivially on $\overline{M}_{g,n}(BG)$. Therefore the  $\mathbb C^*$-equivariant push-forward $$\epsilon_*([\overline{M}_{g,I,\mu_0}([\P^1/G]_{r,1},\mu_\infty)]^{\on{vir}})$$ via the natural map
\[
\epsilon: \overline{M}_{g,I,\mu_0}([\P^1/G]_{r,1},\mu_\infty) \rightarrow \overline{M}_{g,n}(BG)
\]
is a polynomial in $t$. Hence its coefficient of $t^{-1}$ is equal to $0$. 

Set $s=tr$, we will extract the coefficient of $s^0r^0$ in $\epsilon_*(t[\overline{M}_{g,I,\mu_0}([\P^1/G]_{r,1},\mu_\infty)]^{\on{vir}})$. We denote the map
\[
\epsilon: \overline{M}_{g(v),I(v),\mu(v)}(BG(r))\rightarrow \overline{M}_{g(v),n(v)}(BG).
\]
We write
\[
\hat{c}_d=r^{2d-2g(v)+1}\epsilon_* c_d(-L_{g,n}^{1/r})\in A^d(\overline{M}_{g(v),n(v)}(BG)),
\]
then by Proposition \ref{prop:poly}, $\hat{c}_d$ is a polynomial in $r$ for $r$ sufficiently large. So the operation of extracting the coefficient of $r^0$ is valid.
 
 We have
 \begin{align*}
&\epsilon_*(t[\overline{M}_{g,I,\mu_0}([\P^1/G]_{r,1},\mu_\infty)]^{\on{vir}})\\
=&\frac{s}{r}\cdot\sum_\Gamma \frac{1}{|Aut(\Gamma)|}\frac{1}{\prod_{e\in E(\Gamma)}d_e|G_e|}\prod_{(e,v)\in F^S(\Gamma)}\frac{|G_v|}{r_{(e,v)}}\epsilon_*\iota_*\left(\frac{[\overline{M}_{\Gamma}]^{\on{vir}}}{e(N^{\on{vir}})}\right),
\end{align*}
where $\epsilon_*\iota_*\left(\frac{[\overline{M}_{\Gamma}]^{\on{vir}}}{e(N^{\on{vir}})}\right)$ is the product of the following factors:

\begin{enumerate}
\item
For each stable vertex $v\in V(\Gamma)$ over $0$, the factor is
\[
\frac{r}{s}\prod_{e\in E(v)}\frac{r_{(e,v)}}{r}\frac{d_e}{1-\frac{rd_e}{s}\psi_e}\sum_{d\geq 0}\hat{c}_ds^{g(v)-d}.
\]
Each edge contributes a factor $\frac{r_{(e,v)}}{r}$ which cancels with the factor $\frac{|G_v|}{r_{(e,v)}}=\frac{r|G|}{r_{(e,v)}}$ in equation (\ref{localization-equ}) coming from the contribution of the automorphism group of the node labelled by $(k_{(e,v)})^{-1}$. Therefore, we have at least one positive power of $r$ for each stable vertex of the graph over $0$.

\item
When the target degenerates, there is a factor
\[
-\frac{r}{s}\cdot\frac{\prod_{e\in E(\Gamma)}d_e r_{(e,v)}}{1+\frac{r}{s}\psi_\infty}.
\]
We have at least one positive power of $r$ when the target degenerates.
\end{enumerate}

There are only two graphs which have exactly one $r$ factor in the numerator:
\begin{itemize}

\item the graph with a stable vertex of genus $g$ over $0$ and $l(\mu_\infty)$ unstable vertices over $\infty$;

\item the graph with a stable vertex of genus $g$ over $\infty$ and $l(\mu_0)$ unstable vertices over $0$.

\end{itemize}
Therefore, the $r^0$ coefficient is
\begin{align*}
&\on{Coeff}_{r^0}[\epsilon_*(t[\overline{M}_{g,I,\mu_0}([\P^1/G]_{r,1},\mu_\infty)]^{\on{vir}})]\\
=& \frac{|G|^{l(\mu_\infty)}}{|G_e|^{l(\mu_\infty)}}\cdot \on{Coeff}_{r^0}[\sum_{d\geq 0}\hat{c}_ds^{g-d}
]-\frac{|G|^{l(\mu_0)}}{|G_e|^{l(\mu_0)}}DR_g^G(A).
\end{align*}

To extract the coefficient of $s^0$, we take $d=g$,
\[
\on{Coeff}_{r^0s^0}[\epsilon_*(t[\overline{M}_{g,I,\mu_0}([\P^1/G]_{r,1},\mu_\infty)]^{\on{vir}})]
=\frac {|G|^{l(\mu_\infty)}}{|G_e|^{l(\mu_\infty)}}\cdot \on{Coeff}_{r^0}[\hat{c}_g 
]-\frac{|G|^{l(\mu_0)}}{|G_e|^{l(\mu_0)}}DR_g^G(A).
\]
By the vanishing of $\on{Coeff}_{r^0s^0}[\epsilon_*(t[\overline{M}_{g,I,\mu_0}([\P^1/G]_{r,1},\mu_\infty)]^{\on{vir}})]$, we have
\[
DR_g^G(A)=a^{l(\mu_\infty)-l(\mu_0)}\on{Coeff}_{r^0}[r\cdot\epsilon_* c_g(-L_{g,n}^{1/r})]\in A^g(\overline{M}_{g,n}(BG)).
\]
The proof is complete.


\begin{thebibliography}{30}	

\bibitem{af} D. Abramovich, B. Fantechi, \emph{Orbifold techniques in degeneration formulas}, Ann. Sc. Norm. Super. Pisa Cl. Sci. (5) 16 (2016), no. 2, 519--579, arXiv:1103.5132.

\bibitem{JPPZ} F. Janda, R. Pandharipande, A. Pixton, D. Zvonkine, \emph{Double ramification cycles on the moduli spaces of curves}, Publ. Math. Inst. Hautes \'Etudes Sci. 125 (2017), 221--266, arXiv:1602.04705. 

\bibitem{jk} T. Jarvis, T. Kimura, {\em Orbifold quantum cohomology of the classifying space of a finite group}, In: ``Orbifolds in mathematics and physics (Madison, WI, 2001)'', 123--134, Contemp. Math., 310, Amer. Math. Soc. Providence, RI, 2002.

\bibitem{Liu}C.-C. M. Liu, \emph{Localization in Gromov-Witten theory and Orbifold Gromov-Witten Theory,} In: "Handbook of Moduli", Volume II, 353-425, Adv. Lect. Math., (ALM) 25, International Press and Higher Education Press, 2013.

\bibitem{MP} D. Maulik, R. Pandharipande, {\em A topological view of Gromov-Witten theory}, Topology 45 (2006), no. 5, 887--918.


\bibitem{TT0} X. Tang, H.-H. Tseng, {\em Duality theorems of \'etale gerbes on orbifolds}, Adv. Math. 250 (2014), 496--569.

\bibitem{TT1} X. Tang, H.-H. Tseng, {\em Conjugacy classes and characters for extensions of finite groups}, Chin. Ann. Math. Ser. B 35 (2014), no. 5, 743--750.

\bibitem{TT} X. Tang, H.-H. Tseng, \emph{A quantum Leray-Hirsch theorem for banded gerbes,} to appear in J. Differential Geom., arXiv:1602.03564.

\bibitem{Toen} B. To\"en, \emph{Th\'eor\`emes de Riemann-Roch pour les champs de Deligne-Mumford,} K-theory 18 (1999), no. 1, 33--76.

\bibitem{Tseng} H.-H. Tseng, \emph{Orbifold quantum Riemann-Roch, Lefschetz and Serre,}
Geom. Topol. 14 (2010), 1--81. 

\bibitem{TY} H.-H. Tseng, F You, \emph{On orbifold Gromov-Witten theory in codimension one,} J. Pure Appl. Algebra 220 (2016), no. 10, 3567--3571, 
arXiv:1509.02624.

\end{thebibliography}
\end{document}